\documentclass[a4paper, 12pt]{amsart}

\pdfoutput=1
\usepackage{latexsym, amsmath, amssymb, amsthm, mathrsfs, bm, mathtools,comment}
\usepackage[mathscr]{eucal}
\usepackage{array}
\usepackage{tikz}
\usetikzlibrary{arrows,shapes,positioning,calc,patterns,cd}
\usepackage{graphicx} 
\usepackage[T1]{fontenc} 
\usepackage{mathptmx}
\usepackage{microtype}

\usepackage[colorlinks=true,linkcolor=blue,urlcolor=blue, citecolor=blue]{hyperref}
  
  \usepackage[inner=2.3cm,outer=2.3cm, bottom=3.3cm]{geometry}

\usepackage[initials]{amsrefs}
\allowdisplaybreaks

\usepackage{xcolor, color, soul}

\usepackage{color}
\usepackage[all]{xy}  
\usepackage{enumerate}
\usepackage{mathbbol,mathtools}

\renewcommand{\leq}{\leqslant}
\renewcommand{\geq}{\geqslant}

\theoremstyle{definition}
\newtheorem{theorem}{Theorem}[section]
\newtheorem{theoremx}{Theorem}

\newtheorem{question}[theorem]{Question}
\newtheorem{corollary}[theorem]{Corollary}
\newtheorem{lemma}[theorem]{Lemma}
\newtheorem{proposition}[theorem]{Proposition}

\theoremstyle{definition}
\newtheorem{definition}[theorem]{Definition}

\newtheorem{example}[theorem]{Example}

\newtheorem{notation}[theorem]{Notation}

\newtheorem{remark}[theorem]{Remark}
\numberwithin{equation}{subsection}

\newcommand{\m}{\mathfrak{m}}
\newcommand{\n}{\mathfrak{n}}
\newcommand{\fpt}{{\operatorname{fpt}}}
\newcommand{\mpt}{{\operatorname{mpt}}}

\newcommand{\mpdim}{{\operatorname{mpdim}}}

\newcommand{\RR}{\mathbb{R}}
\newcommand{\NN}{\mathbb{Z}_{\geqslant 0}}
\newcommand{\ZZ}{\mathbb{Z}}
\newcommand{\QQ}{\mathbb{Q}}

\newcommand{\FF}{\mathbb{F}}
\newcommand{\LL}{\mathbb{L}}
\newcommand{\kk}{\mathbb{k}}
\newcommand{\bx}{\mathbf{x}}
\newcommand{\balpha}{\mathbf{\alpha}}

\newcommand{\cF}{\mathcal{F}}

\newcommand{\cP}{\mathcal{P}}
\newcommand{\cA}{\mathcal{A}}
\newcommand{\fp}{\mathfrak{p}}

\newcommand{\Fcp}{\mathcal{F}_M}

\newcommand{\rank}{\operatorname{rank}}
\newcommand{\reg}{\operatorname{reg}}
\newcommand{\Reg}{\operatorname{Reg}}
\newcommand{\pd}{\operatorname{pd}}
\newcommand{\sdim}{\operatorname{sdim}}

\newcommand{\Depth}{\operatorname{depth}}
\newcommand{\Hom}{\operatorname{Hom}}

\newcommand{\Ext}{\operatorname{Ext}}
\newcommand{\Tor}{\operatorname{Tor}}

\newcommand{\Ker}{\operatorname{Ker}}
\newcommand{\mpr}{\operatorname{mpr}}

\newcommand{\Ann}{\operatorname{Ann}}
\newcommand{\inte}{\operatorname{int}}

\newcommand{\Char}{\operatorname{char}}

\newcommand{\vol}{\operatorname{vol}}

\newcommand{\rk}{\operatorname{rank}}	
\newcommand{\Ht}{\operatorname{ht}}

\newcommand{\FDer}[1]{\stackrel{#1}{\to}}

\newcommand{\ls}{\leqslant}%
\newcommand{\gs}{\geqslant}

\newcommand{\mps}{\operatorname{mps}}

\makeatletter
\@addtoreset{equation}{section}
\makeatother

\author[A. De Stefani]{Alessandro De Stefani$^1$}
\address{Dipartimento di Matematica, Universit{\`a} di Genova, Via Dodecaneso 35, 16146 Genova, Italy}
\email{destefani@dima.unige.it}
\thanks{$^{1}$ The first author was partially supported by the PRIN 2020 project 2020355B8Y ``Squarefree Gröbner degenerations, special varieties and related topics".}

\author[J. Monta{\~n}o]{Jonathan Monta{\~n}o$^2$}
\address{School of Mathematical and Statistical Sciences, Arizona State University, P.O. Box 871804, Tempe, AZ 85287-18041}
\email{montano@asu.edu}
\thanks{$^{2}$ The second author was  supported by NSF Grant DMS \#2001645/2303605.}

\author[L. N{\'u}{\~n}ez-Betancourt]{Luis N{\'u}{\~n}ez-Betancourt$^3$}
\address{Centro de Investigaci{\'o}n en Matem{\'a}ticas, Guanajuato, Gto., M{\'e}xico}
\email{luisnub@cimat.mx}
\thanks{$^{3}$ The third author was supported by CONACyT Grant \#284598.}



\begin{document}
\newcommand{\tens}{\otimes}
\newcommand{\hhtest}[1]{\tau ( #1 )}
\renewcommand{\hom}[3]{\operatorname{Hom}_{#1} ( #2, #3 )}

\title{Purity of monoids and characteristic-free splittings in semigroup rings}

\begin{abstract} 
Inspired by methods in prime characteristic in commutative algebra, we introduce and study combinatorial invariants of seminormal monoids. We relate such numbers with the singularities and homological invariants of the semigroup ring associated to the monoid. Our results are characteristic independent.
\end{abstract}

\maketitle

\maketitle
\setcounter{secnumdepth}{1}
\setcounter{tocdepth}{1}
 \tableofcontents
 
\section{Introduction}

Frobenius splittings have inspired a large number of results in commutative algebra, algebraic geometry, and representation theory. In this manuscript we seek to continue this approach in the context of  combinatorics of monoids. 
Given a monoid $M\subseteq \ZZ^q_{\gs 0}$ for some $q\in \ZZ_{>0}$,  and $m\in\ZZ_{>0}$, we study the pure $M$-submodules of $\frac{1}{m}M$ that are translations of $M$, which algebraically corresponds to free summands of $\kk[\frac{1}{m}M]$ as $\kk[M]$-module. 
It turns out that the purity of $M\subseteq \frac{1}{m}M$ detects both normality and seminormality (see Proposition \ref{redefnormal}).
The study of pure submodules, or equivalently of free summands, of normal monoids was already initiated by other authors in order to compute the $F$-signature of normal affine semigroup rings   \cite{ToricSingh,von2011f}. Moreover, the structure of $\frac{1}{m}M$ as $M$-module was described by Bruns and Gubeladze \cite{BGSemigroups,BrunsDivisorNormalMonoid} for normal monoids (see \cite{Shibuta} for a related result in prime characteristic).

In this manuscript we study combinatorial numerical invariants of a seminormal monoid. 
Our key motivation is that seminormality for a monoid can be seen as  a characteristic-free version of $F$-purity for affine semigroup rings. For more information and examples on seminormal monoids we direct the interested reader to Li's thesis on this subject \cite{Li}.

In Definition \ref{mstDef} we introduce the  notion of {\it  pure threshold} of a seminormal monoid $M$, denoted by $\mpt(M)$, which is motivated by the $F$-pure threshold in prime characteristic. This number can be described as the largest degree  of a pure translation of $M$ inside the cone $\RR_{\gs 0}M$ or, equivalently, of $\frac{1}{m}M$ for some $m$. 
We show that $\mpt(M)$ gives an upper bound for the Castelnuovo-Mumford regularity $\reg(\kk[M])$ defined in terms of local cohomology, and the Castelnuovo-Mumford regularity $\Reg(\kk[M])$ defined in terms of graded Betti numbers of $\kk[M]$ (see Section \ref{Background} for more details).

\begin{theoremx}[{Theorem \ref{ThmRegMPT}}]\label{MainThmSST}
Let $M$ be a seminormal monoid with a minimal set of generators  $\{\gamma_1,\ldots, \gamma_u\}$. Then,
$a_i(\kk[M])\leq -\mpt(M)$. As a consequence,
\[
\reg(\kk[M])=\max\{a_i(\kk[M])-i\}\leq \dim(\kk[M]) - \mpt(M)= \rank(M) - \mpt(M).
\]
 Moreover, if we present $R$ as $S/I$, where $S=\kk[x_1,\ldots,x_u]$ and each $x_i$ has degree $d_i:=\deg(x_i)=|\gamma_i|$ the degree of $\gamma_i$ for $i=1,\ldots,u$,   and $I \subseteq S$ is a homogeneous ideal, then 
\[
\Reg(\kk[M])= \sup\{\beta_i^S(M) - i \mid i \in \ZZ\} \leq  \rank(M) + \sum_{i=1}^u(d_i-1)- \mpt(M).
\] 
\end{theoremx}
Theorem \ref{MainThmSST} allows us to give an upper bound for the degrees of generators of the defining ideal $I$. We also show that $\mpt(M)$  is a rational if $M$ is a normal  (see Proposition \ref{mstRat}). Despite $\mpt(M)$ being inspired by $F$-pure thresholds, these  numbers do not always coincide (see Example \ref{ExVeronese} and Remark \ref{remark mpt VS fpt}). In addition, $\mpt(M)$ is defined independently of the field $\kk$ and so it is a characteristic-free invariant, while the $F$-pure threshold is only defined when $\kk$ has prime characteristic.

We introduce the {\it pure prime ideal} $\cP(M)$, and the {\it pure prime face}  $\cF_M$, of a  seminormal monoid $M$ (see Corollary \ref{FaceP}, and Definitions \ref{indeedprime} and \ref{msd}). 
The former emulates the splitting prime ideal of an $F$-pure ring, while the latter is related to the quotient of a ring by its splitting prime. In fact, the  submonoid $\cF_M \cap M$ is normal (see Corollary \ref{FaceP}). We note that  the rank of $\cF_M \cap M$  is a monoid version of the splitting dimension and so we call it the {\it pure dimension} and denote it by $\mpdim(M)$. It turns out that this rank is  equal to the rank of $M$ if and only if $M$ is normal, and it is non-negative if and only if $M$ is seminormal  (see Corollary \ref{msdCor}). Therefore, in some sense, $\mpdim(M)$ measures how far a seminormal monoid is from being normal. Furthermore, $\mpdim(M)$ is related to the depth of $\kk[M]$ as the following theorem shows.

 \begin{theoremx}[{Theorem \ref{ThmDepthMPT}}]\label{MainThmDepth}
If $M$ is a seminormal monoid, then 
$\mpdim(M)\leq \Depth(\kk[M]).$
\end{theoremx}
\noindent We point out that Theorem \ref{MainThmDepth} recovers Hochster’s result that normal semigroup rings are Cohen-Macaulay \cite{HochsterToric}. 

Finally, we consider the growth of the number of disjoint pure translations of $M$ in $\frac{1}{m}M$ as $m$ varies. More specifically, if $m \in \ZZ_{>0}$ is such that $\frac{1}{m}M \cap \ZZ M = M$, we define  $$V_m(M) := \left\{\alpha \in \frac{1}{m}M \mid (\alpha + M) \subseteq \frac{1}{m}M \text{ is pure}\right\}.$$

\begin{theoremx}[{Theorem \ref{ratioExists}}]\label{MainThmRatio}
Let $M$ be a seminormal monoid, $\cA(M)=\{m\in\ZZ_{>0} \; | \; \frac{1}{m}M\cap \ZZ M=M\}$, and $s=\mpdim(M)$. Then,
$$
\mpr(M) := \lim\limits_{t\to\infty} \frac{|V_{m_t}(M)|}{m_t^s}
$$
exists and it is positive for every increasing sequence $m_t\in \cA(M)$.
Furthermore, if $M$ is normal, then  $\mpr(M)\in \QQ_{>0}$.
\end{theoremx}

We call the limit in Theorem \ref{MainThmRatio}  the {\it pure ratio} of $M$. If the field has prime characteristic, this number coincides with the splitting ratio 
of $\kk[M]$ \cite{AE}. A consequence of Theorem \ref{MainThmRatio} is that the value of the $F$-splitting ratio depends only on the structure of $M$, and so it is independent of the characteristic of the field as long as $\kk[M]$ is $F$-pure. Finally, using this result we give a monoid version of a celebrated Theorem of Kunz \cite[Theorem 2.1]{Kunz} which  characterizes regularity of rings of prime characteristic in terms of Frobenius (see Theorem \ref{ThmKunz}).

Throughout this article we  adopt the following notation.

\begin{notation}\label{notation}
Let $\kk$  be  a field of any characteristic and $q$ a positive integer. Let $M\subseteq \ZZ^q_{\gs 0}$ be an affine monoid, i.e., a finitely generated submonoid of $\ZZ^q$. We fix $\{\gamma_1,\ldots, \gamma_u\}$ a minimal set of generators of $M$.  Let $\ZZ M$ denote  the group generated by $M$ and $C(M) = \RR_{\gs 0}  M$  the cone generated by $M$. 
\end{notation}
\section{Background}\label{Background}

In this section we include some preliminary information that is needed in the rest of the paper.

\subsection{Affine monoids and affine semigroup rings} 
For proofs of the claims in this subsection and further information about affine monoids we refer the reader to Bruns and Gubeladze's book \cite{bruns2009polytopes}.
 Let $M\subseteq \ZZ_{\gs 0}^q$ be an affine monoid. A subset $U\subseteq \QQ^q$ is an {\it $M$-module} if $U+M\subseteq U$. 
An $M$-module $U$ is an {\it ideal} if it is contained in $M$. An ideal $U\subseteq M$ is {\it prime} if whenever $a+b\in U$ with $a,b\in M$, we must have $a\in U$ or $b\in U$. The {\it rank} of $M$ is the dimension of the  $\QQ$-vector space $\QQ\otimes_{\ZZ} \ZZ M$. 

Let $R=\kk[M]\subseteq \kk[\bx]:=\kk[x_1,\ldots, x_q]$ be the affine semigroup associated to $M$. As a $\kk$-vector space, $R$ is generated by the monomials 
$\{x^\alpha\mid \alpha\in M\}$. 
We note that the monomial ideals of $R$ are precisely those generated by $\{x^\alpha\mid \alpha\in U\}$ for some ideal $U\subseteq M$.
 Under this correspondence, prime monomial ideals of $R$ correspond to prime ideals of $M$. For every $M$-module 
 $U\subseteq \QQ^q$ we have a corresponding $R$-module
 $R U:= \{x^{\alpha + \eta}\mid \alpha\in M, \eta\in U\}$ in the algebraic closure of $\kk(x_1,\ldots, x_q)$. Moreover, we have $\dim(R)=\rank(M)$.

\subsection{Graded algebras and modules} A non-negatively graded algebra $A$ is a ring that admits a direct sum decomposition $A= \bigoplus_{j \geq 0} A_j$ of Abelian groups such that $A_i \cdot A_j \subseteq A_{i+j}$. It follows from this that $A_0$ is a ring, and each $A_i$ is an $A_0$-module. If we let $A_+ = \bigoplus_{j > 0} A_j$, then $A_+$ is an ideal of $A$, called the irrelevant ideal. 

Throughout this manuscript we will make the assumption that $A$ is Noetherian or, equivalently, that there exist finitely many elements $a_1,\ldots,a_n \in A_+$ such that $A=A_0[a_1,\ldots,a_n]$, which can be assumed to be homogeneous of degrees $d_1,\ldots,d_n$. In this case, note that $A$ is a quotient of a polynomial ring $A_0[x_1,\ldots,x_n]$ by a homogeneous ideal.

A $\ZZ$-graded $A$-module is an $A$-module $N$ that admits a direct sum decomposition $N=\bigoplus_{j \in \ZZ}N_j$ of Abelian groups, and such that $A_i \cdot N_j \subseteq N_{i+j}$. As a consequence, each $N_i$ is an $A_0$-module. Moreover, if $N$ is Noetherian there exists $i_0 \in \ZZ$ such that $N_i = 0$ for all $i<i_0$; on the other hand, if $N$ is Artinian there exists $j_0 \in \ZZ$ such that $N_j=0$ for all $j>j_0$.

Given a $\ZZ$-graded $A$-module $N$, and an integer $j \in \ZZ$, we define the shift $N(j)$ as the $\ZZ$-graded $A$-module whose $i$-th graded component is $ N(j)_i=N_{i+j}$. In particular, $A(-j)$ is a free graded $A$-module of rank one with generator in degree $j$.

\subsection{Graded local cohomology and Castelnuovo-Mumford regularity}
In this subsection we recall general properties of local cohomology. We refer the interested reader to Brodmann and Sharp's book on this subject  \cite{BrodmannSharp}.
Let $\kk$ be a field and $S = \kk[x_1,\ldots,x_n]$, with $\deg(x_i) = d_i>0$. Let $N$ be a finitely generated $\ZZ$-graded $S$-module. If we let $\m=(x_1,\ldots,x_n)$, then the graded local cohomology modules $H^i_\m(N)$ are Artinian and $\ZZ$-graded.
\begin{definition} 
Let $a_i(N) = \sup\{j \in \ZZ \mid H^i_\m(N)_j \ne 0\}$ be the $i$-th $a$-invariant of $N$. If $N \ne 0$ we define the Castelnuovo-Mumford regularity of $N$ as $\reg(N) = \sup\{a_i(N) + i \mid i=0,\ldots,n\}$. On the other hand, if $N=0$ we let $\reg(N) = -\infty$. 
\end{definition}

In the standard graded case, $\reg(N)$ has a well-known interpretation in terms of graded Betti numbers of $N$. In our setup this is still the case, but the degrees of the algebra generators of $S$ must be taken into account. 
For $i \in \NN$, let $\beta_i^S(N) = a_0(\Tor_i^S(N,\kk)) \in \ZZ \cup \{-\infty\}$. As another way to see this, for a non-zero graded $S$-module $T$ let $\beta(T)$ be the maximum degree of an element in a minimal homogeneous generating set of $T$. If
\[
\xymatrixcolsep{10mm}
\xymatrix{
F_\bullet: 0 \ar[r] & F_c \ar[r] & F_{c-1}  \ar[r] &\ldots \ldots \ar[r] & F_1 \ar[r] & F_0 \ar[r] & N \ar[r] & 0,
}
\]
is a minimal graded free resolution of $N$ where $c:=\pd(N)$ is the projective dimension of $N$, then $\beta_i^S(N) = \beta(F_i)$.

\begin{definition}
For $N \ne 0$ we let $\Reg(N) = \sup\{\beta_i^S(N) - i \mid i \in \ZZ\}$, while for $N=0$ we let $\Reg(N) = -\infty$.
\end{definition}
In the standard graded case, that is, when $d_1=\ldots = d_n=1$, then $\reg(N)= \Reg(N)$. In our more general scenario, we still have the following relation between the two notions of regularity. 
\begin{lemma} \label{Lemma regularity} With the above notation we have that
\[
\Reg(N) = \reg(N) + \sum_{i=1}^n(d_i-1).
\]
\end{lemma}
\begin{proof}
We may assume that $n>0$, otherwise the claim is trivial. We prove the statement by induction on $c=\pd_S(N)$. If $c=0$ then $N$ is free, and it is clear that $\Reg(N) = \beta(N)$. On the other hand, $a_i(N) = 0$ for all $i \ne n$, while $a_n(N) = a_n(S) + \beta(N) = -\sum_{i=1}^n d_i + \beta(N)$. It follows that $\reg(N) = n+a_n(N) = \beta(N) + \sum_{i=1}^n (1-d_i) = \Reg(N) + \sum_{i=1}^n(1-d_i)$, and the base case follows. 

Now assume that $c \geq 1$. We have a graded short exact sequence $0 \to \Omega \to F_0 \to N \to 0$, where $F_0$ is the first free module in a minimal free resolution of $N$, and $\pd_S(\Omega) = c-1$. By induction we have that $\Reg(\Omega) = \reg(\Omega) + \sum_{i=1}^n(d_i-1)$. Moreover, it is clear from the definitions that $\Reg(N) = \max\{\beta(N),\Reg(\Omega)-1\} = \max\{\beta(N),\reg(\Omega)+\sum_{i=1}^n(d_i-1)-1\}$.  We have that $\beta(N) \leq \reg(N) + \sum_{i=1}^n (d_i-1)$ \cite[Proposition 3.1]{DSMNB}, and therefore the above equality gives that 
\begin{equation}
\label{ineq} \Reg(N) \leq \max\{\reg(N),\reg(\Omega)-1\} + \sum_{i=1}^n (d_i-1).
\end{equation}
We now show that $\max\{\reg(N),\reg(\Omega)-1\} = \reg(N)$. The short exact sequence $0 \to \Omega \to F_0 \to N \to 0$ yields graded isomorphisms $H^i_\m(N) \cong H^{i+1}_\m(\Omega)$ for all $i<n-1$, and a graded exact sequence $0 \to H^{n-1}_\m(N) \to H^n_\m(\Omega) \to H^n_\m(F_0) \to H^n_\m(N) \to 0$. If we had $\reg(\Omega)-1 > \reg(N)$, then necessarily $\reg(\Omega) = a_n(\Omega)+n$, and looking at top degrees in the above exact sequence we also conclude that $a_n(\Omega) = a_n(F_0)$. On the other hand, $a_n(F_0) = \beta(N) + a_n(S) = \beta(N) - \sum_{i=1}^n d_i \leq \reg(N) - n$  \cite[Proposition 3.1]{DSMNB}, and so, $\reg(\Omega) \leq \reg(N)$, a contradiction. Thus we always have that $\reg(\Omega)-1 \leq \reg(N)$, and by (\ref{ineq}) the inequality $\Reg(N) \leq \reg(N) + \sum_{i=1}^n (d_i-1)$ is proved.

For the reverse inequality, first observe that the above isomorphisms give that $a_i(N) = a_{i+1}(\Omega)$ for all $i<n-1$, while the exact sequence yields that $a_{n-1}(N) \leq a_n(\Omega)$. Since $n>0$ and $\Omega$ is a submodule of a free module, it has positive depth, and thus $a_0(\Omega) = 0$. It follows that $\max\{a_i(N)+i \mid i=0,\ldots,n-1\} \leq \max\{a_{i+1}(\Omega)+i \mid i=0,\ldots,n-1\} = \reg(\Omega)-1$. By induction we have that $\reg(\Omega) = \Reg(\Omega)+\sum_{i=1}^n(1-d_i)$, and thus $\max\{a_i(N)+i \mid i=0,\ldots,n-1\} \leq \Reg(\Omega)-1+\sum_{i=1}^n(1-d_i) \leq \Reg(N)+\sum_{i=1}^n(1-d_i)$ since the inequality $\Reg(\Omega)-1 \leq \Reg(N)$ always holds. Now, the above exact sequence on local cohomology also gives that $a_n(N) \leq a_n(F_0) = \beta(N) + a_n(S) \leq \Reg(N) - \sum_{i=1}^n d_i$, and thus $a_n(N)+n \leq \Reg(N) + \sum_{i=1}^n (1-d_i)$. In conclusion, we have that $\reg(N) = \sup\{a_i(N) + i \mid i=0,\ldots,n\} \leq \Reg(N) + \sum_{i=1}^n (1-d_i)$, and the proof is complete.
\end{proof}

\section{Purity of $M$-modules and (semi)normal affine monoids}\label{SplitSection}

\begin{definition}\label{splittDef}
Let $U\subseteq V\subseteq \QQ^q$ be $M$-modules. We say that the inclusion   $U\subseteq V $ is {\it pure} if $V\setminus U$ is also an $M$-module.
\end{definition}

\begin{example}
Let $M=\ZZ_{\gs 0}(2,1)+\ZZ_{\gs 0}(1,2)$ be the monoid generated by $\{(2,1), (1,2)\}$. We have that the inclusion $(\frac32, \frac32)+M\subset \frac12 M\subset \QQ^2$ is pure. In Figure \ref{pureInclusion} we represent the elements of $M$ with circles and the ones from $(\frac32, \frac32)+M$ with multiplication signs.  The shaded region is included to illustrate that $(\frac32, \frac32)+M$ is obtained as a translation of $M$.
\end{example}

\begin{figure}[ht]
\begin{center}
\includegraphics[width=0.4\textwidth]{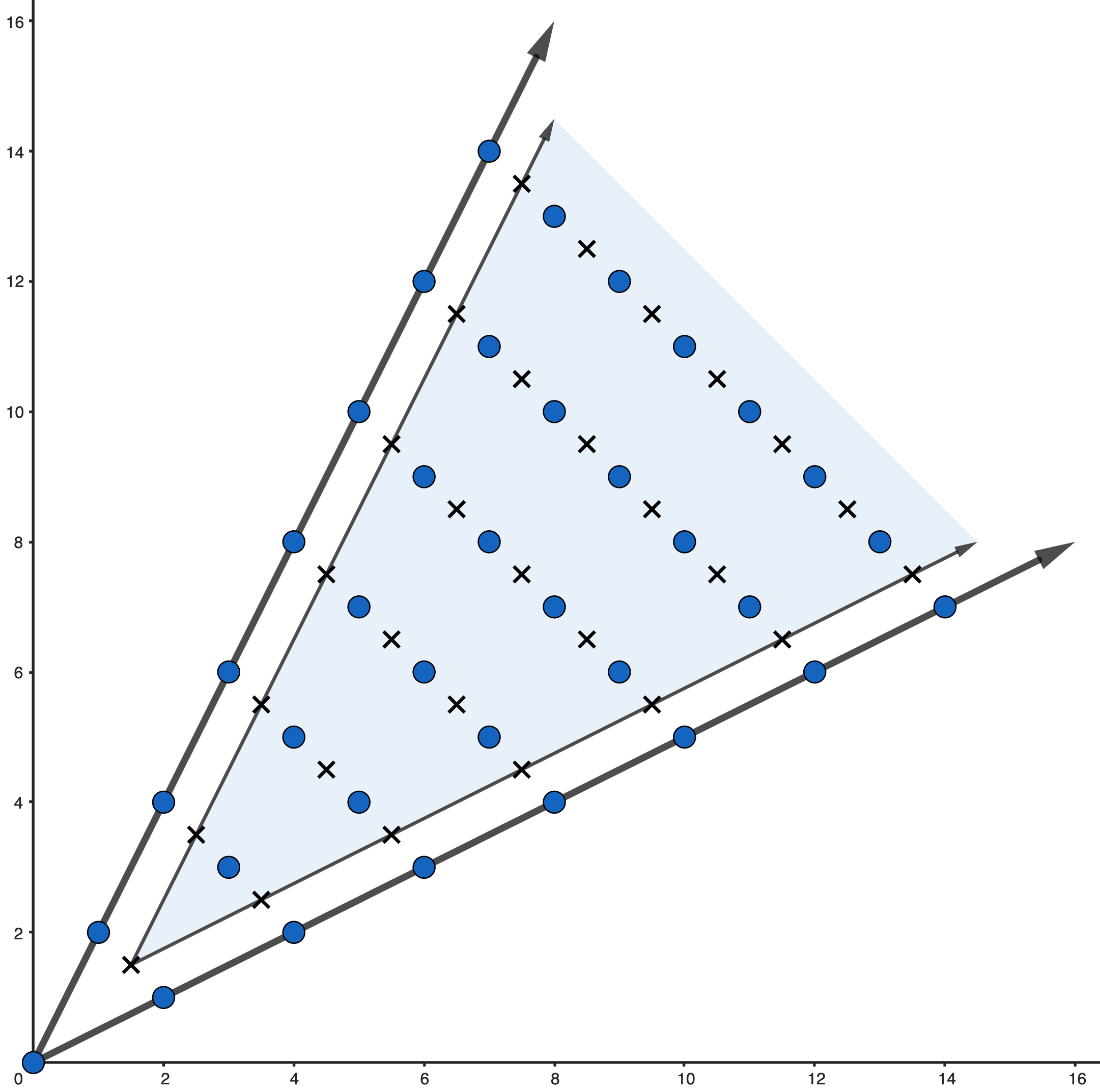}
\end{center}
\caption{For $M=\ZZ_{\gs 0}(2,1)+\ZZ_{\gs 0}(1,2)$, the inclusion $(\frac{3}{2},\frac{3}{2})+M\subset \frac{1}{2}M\subset \QQ^2$ is pure.}
	\label{pureInclusion} 
\end{figure}

In the following proposition, we provide  equivalent statements for Definition \ref{splittDef} in a particular case.

\begin{proposition}\label{reinterpret}
Let  $V\subseteq \QQ^q$ be an $M$-module and $\alpha\in V$. The following statements are equivalent:
\begin{enumerate}
\item[(i)] The inclusion  $(\alpha +M)\subseteq V$ is pure.
\item[(ii)] For every $\gamma \in M$ and $\beta\in V$ we have $\gamma +\beta -\alpha \in M \text{ implies }\beta -\alpha \in M.$
\item[(iii)]  $\left(V-\alpha\right)\cap \ZZ M\subseteq M $
\end{enumerate}
\end{proposition}
\begin{proof}
First,  assume (i)  and let $\gamma \in M$ and $\beta\in V$ with $\gamma+\beta\in \alpha +M$. Thus, $\beta \in \alpha + M$ and then  (ii) follows.

Now, assume (ii)  and  let $\beta\in V$ be such that $\beta-\alpha\in \ZZ M$. Write $\beta-\alpha=\sigma-\gamma$ with $\sigma, \gamma\in M$. Then $\gamma+\beta-\alpha=\sigma\in M$ implies $\beta-\alpha\in M$. Thus, (iii)  follows.

Finally, assume (iii). Let $\beta\in V\setminus (\alpha+M)$ and $\gamma \in M$. Assume by means of contradiction that $\beta+\gamma \in \alpha +M$. Then  $\beta+\gamma=\alpha +\sigma$ for some $\sigma\in M$, which implies  $\beta-\alpha \in \left(V-\alpha\right)\cap \ZZ M\setminus M$ which contradicts (iii). Thus, $\beta+\gamma \in V\setminus (\alpha+M)$ and then (i) follows. 
\end{proof}

We now discuss seminormality and normality, which are the main subjects of study in this manuscript. We refer to the work of Bruns, Li and R{\"o}mer   \cite{BrunsLiRomer} to reader in seminormal rings. 

\begin{definition} 
A monoid $M$ is called {\it seminormal} if, whenever $\alpha \in \ZZ M$ is such that $2 \alpha \in M$ and $3 \alpha \in M$, then $\alpha \in M$. The monoid is called {\it normal} if $C(M) \cap \ZZ M = M$.
\end{definition}

The following alternative characterization of seminormality and normality is useful for the proof of our main results. While it might be already known to experts, we record it here with a proof for convenience of the reader. For a related result in prime characteristic we refer to the work of Bruns, Li, and  R\"{o}mer \cite[Section 6]{BrunsLiRomer}.

\begin{proposition}\label{redefnormal}
Let $M$ be an affine monoid.
\begin{enumerate}
\item  $M$ is seminormal  if and only if there exists $m\in \ZZ_{>1}$ such that $\frac{1}{m}M\cap \ZZ M= M$.
\item  $M$ is normal  if and only if  $\frac{1}{m}M\cap \ZZ M= M$ for every $m\in \ZZ_{>1}$.
\end{enumerate}
\end{proposition}
\begin{proof}
We note that the containment  $\frac{1}{m}M\cap \ZZ M\supseteq M$  holds trivially for any $m\in \ZZ_{>1}$.

For (1), assume that $\frac{1}{m}M \cap \ZZ M \subseteq M$ for some $m \in \ZZ_{>1}$. Let $\alpha \in \ZZ M$ be such that $2\alpha \in M$ and $3 \alpha \in M$. We can find nonnegative integers $a$ and $b$ such that $m=2a+3b$, and thus $m\alpha = a(2\alpha)+b(3\alpha) \in M$. It follows by our assumption that $\alpha \in M$, and thus $M$ is seminormal. Conversely, let $\mathcal{F}$ be the set of all faces of $C(M)$  (of any dimension); we note that $\mathcal{F}$ is a finite set. For any $F \in \mathcal{F}$ we consider the finitely generated Abelian group $G_F = (\ZZ M \cap \QQ  F)/(\ZZ(M \cap F))$. Let $p \gg 0$ be a prime number such that the ideal $(p)$  is not associated to $G_F$ as a $\ZZ$-module for any $F \in \mathcal{F}$. We claim that $\frac{1}{p}M \cap \ZZ M \subseteq M$. Let $\alpha \in \frac{1}{p}M \cap \ZZ M\subseteq C(M)\cap \ZZ M$, then $\alpha\in \inte(F')$ the interior of some  $F'\in \mathcal{F}$, and also  $p\alpha\in M\cap F'$. By the choice of $p$ it follows that $\alpha\in \ZZ(M\cap F')\cap \inte(F')$ and then $\alpha\in M$ \cite[Theorem 2.1]{BrunsLiRomer}.

For (2),  if $M$ is normal then  $\frac{1}{m}M \cap \ZZ M\subseteq C(M)\cap \ZZ M=M$. Conversely, let $\alpha \in C(M) \cap \ZZ M$, then $\alpha = \sum_{i=1}^r \frac{t_i}{m_i} \alpha_i$, with $\alpha_i \in M$ and $\frac{t_i}{m_i} \in \QQ_{\gs 0}$. If we let $m=\prod_{i=1}^r m_i$, it then follows that $\alpha \in \frac{1}{m}M \cap \ZZ M = M$, as desired.
\end{proof}


 Motivated by the previous result, we consider the following definition.

\begin{definition}
 We set $\cA(M)=\{m\in \ZZ_{>1}\; | \; \frac{1}{m}M\cap \ZZ M= M\}$.
\end{definition}

\begin{remark}\label{allmNorm}
As a consequence of Proposition \ref{redefnormal}, we deduce that
\begin{enumerate}
\item $\cA(M)\neq \emptyset $ if and only if $M$ is seminormal;
\item $M$ is normal if and only if $\cA(M)=\ZZ_{>1}$.
\end{enumerate}
\end{remark}

We now see that $\cA(M)$ is a multiplicative set.

\begin{lemma}\label{mnboth}
Let $m,n\in \ZZ_{>1}$. Then $mn \in \cA(M)$ if and only if both $m \in \cA(M)$ and $n \in \cA(M)$.
\end{lemma}
\begin{proof}
First assume that $mn \in \cA(M)$. Then $m \in \cA(M)$ because $\frac{1}{m}  M\cap \ZZ M \subseteq \frac{1}{mn}M \cap \ZZ M = M$. Likewise $n \in \cA(M)$. For the converse, assume that $\alpha \in \ZZ M$ is such that $mn\alpha \in M$. Note that $\beta = n\alpha \in \ZZ M$ is such that $m \beta \in M$, and thus $\beta \in M$ because $m \in \cA(M)$. But then $\alpha \in \ZZ M$ is such that $n\alpha \in M$, and thus $\alpha \in M$ because $n \in \cA(M)$. It follows that $mn \in \cA(M)$.
\end{proof}

Now, we consider a set that records the pure translations of $M$ in $\frac{1}{m} M$. 
In Section \ref{SectionRings} we see that this set corresponds to the free summands of $\kk\left[\frac{1}{m}M\right]$ as a $\kk[M]$-module.

\begin{definition}\label{defObj}
Let $m\in\ZZ_{> 1}$. 
We set 
$$V_m(M)=\left\{\alpha\in\frac{1}{m}M \;\bigg|\;  \left(\frac{1}{m}M-\alpha\right)\cap \ZZ M\subseteq M \right\}.$$
Moreover, we set 
$$V(M)=\bigcup_{m \in \ZZ_{> 1}}V_m(M).$$
\end{definition}

\begin{remark}\label{remDefVm}
We note that, because of Proposition \ref{reinterpret}, $V_m(M)$ is precisely the set of $\alpha\in \frac{1}{m}M$ such that $(\alpha + M)\subseteq \frac{1}{m}M$ is pure. 
\end{remark}

\begin{remark}\label{Vmnotempty}
We note that $V_m(M)\neq \emptyset$ if and only if  $m\in\cA(M)$. As a consequence  we have  $$V(M)=\bigcup_{m \in  \cA(M)}V_m(M).$$ Indeed, if $m\in \cA(M)$ then $0\in V_m(M)$. Conversely,  fix $\alpha\in V_m(M)$. From  the containments $\frac{1}{m}M\cap \ZZ M= \left(\left(\frac{1}{m}M+\alpha\right)-\alpha\right)\cap \ZZ M \subseteq\left(\frac{1}{m}M-\alpha\right)\cap \ZZ M \subseteq M$, it follows that $m\in \cA(M)$.

\end{remark}

\begin{remark}\label{mnotvm} 
If $M$ is seminormal, then  $M\cap V(M)=\{0\}$. Indeed, if $0\neq \gamma\in M$, then $-\gamma\in \left(\left(\frac{1}{m}M-\gamma\right)\cap \ZZ M\right) \setminus M$ for every $m\in \cA(M)$.
\end{remark}

In the following remarks we observe that  $V_m(M)$ is compatible with projections onto faces of $C(M)$ and with isomorphisms of monoids. 

\begin{remark}\label{VmProjFace}
For every face $F$  of $C(M)$ and $m\in \ZZ_{\gs 1}$ we have 
$V_m(M)\cap F\subseteq V_m(M\cap F)$. Indeed, let $\alpha\in V_m(M)\cap F$, then 
$$ \left(\frac{1}{m}(M\cap F)-\alpha\right)\cap \ZZ (M\cap F)\subseteq \left(\left(\frac{1}{m}M-\alpha\right)\cap \ZZ M\right)\cap F=M\cap F.$$
\end{remark}

\begin{remark}\label{VmIsom}
Since every isomorphism of monoids $\varphi: M\to M'$ extends to an isomorphism of groups $\varphi: \ZZ M\to \ZZ M'$, it follows that $\cA(M')=\cA(M)$. Furthermore,  for every $m\in \cA(M)$ we have $mV_m(M')=\varphi(mV_m(M))$.
\end{remark}

We now describe basic properties $V_m(M)$. In particular, we see that this set is finite.

\begin{proposition}\label{minusmbelong}
Let $m\in \ZZ_{> 1}$ and  $\alpha,\beta, \eta\in \frac{1}{m}M$ be such that $\alpha=\beta+\eta$. If $\alpha\in V_m(M)$, then $\beta, \eta\in V_m(M)$.
\end{proposition}
\begin{proof}
 Let $w\in \frac{1}{m}M$ be such that $w-\beta\in \ZZ M$. Thus, from 
$$w-\beta=(w+\eta)-\alpha  \in \left(\frac{1}{m}M-\alpha\right)\cap \ZZ M\subseteq M$$
it follows that $w-\beta\in M$. This argument implies that   $\beta \in V_m(M)$. Likewise, $\eta \in V_m(M)$,  finishing the proof.
\end{proof}

\begin{corollary}\label{compIdeal}
For every $m\in \ZZ_{> 1}$ the set $M\setminus mV_m(M)$ is an ideal of $M$.
\end{corollary}
\begin{proof}
Let  $a\in M\setminus mV_m(M)$ and $g\in M$. Suppose $a+g\in mV_m(M)$, then $\frac{a+g}{m}\in V_m(M)$. By Proposition \ref{minusmbelong} this implies $\frac{a}{m}\in V_m(M)$ which is a contradiction. 
\end{proof}

The following lemma provides  useful facts about the sets $V_m(M)$. We recall that the {\it Minkowski} sum of two subsets $A,B\subseteq \RR^q$ is defined as $A\oplus B=\{a+b\mid a\in A,b\in B\}$.

\begin{lemma}\label{vminvmn}
Let $m,n\in\cA(M)$.
Then,
\begin{enumerate}
\item $\frac{1}{m}M\cap V_{mn}(M)=V_m(M)$.
\item $ \frac{1}{n}V_m(M)\oplus V_n(M)\subseteq  V_{mn}(M)$;
\item $|\frac{1}{n}V_m(M)\oplus V_n(M)|=|V_m(M)|\cdot |V_n(M)|\ls |V_{mn}(M)|$.
\end{enumerate}

\end{lemma}
\begin{proof}

We begin with the containment $\supseteq$ in (1).  Let $\alpha\in V_m(M)$, then $\alpha \in \frac{1}{m}M\subseteq \frac{1}{mn}M$. Consider $y\in\left(\frac{1}{mn}M-\alpha\right)\cap \ZZ M $, then $my+m\alpha\in \frac{1}{n}M \cap \ZZ M=M$,
Thus,  $y\in \left(\frac{1}{m}M-\alpha\right)\cap \ZZ M\subseteq M$ and the conclusion follows. Now we prove the containment $\subseteq$. Let $\alpha\in \frac{1}{m}M\cap V_{mn}(M)$ and let $\beta\in \frac{1}{m}M$ be such that $\beta-\alpha\in \ZZ M$. Then $\beta-\alpha\in \left(\frac{1}{m}M-\alpha\right)\cap \ZZ M\subseteq \left(\frac{1}{mn}M-\alpha\right)\cap \ZZ M\subseteq M$.


We continue with (2). Let $\alpha\in V_m(M)$ and  $\beta\in V_n(M)$. Clearly we have $\frac{1}{n}\alpha+\beta\in\frac{1}{mn}M$. Let $\gamma\in M$ and $\eta\in \frac{1}{mn}M$ be such that $\gamma+\eta-(\frac{1}{n}\alpha+\beta)\in M$, then  $n\gamma+n\eta-\alpha\in M$. Since $n\eta \in \frac{1}{m}M$, Proposition \ref{reinterpret} applied to $\alpha$ implies $n\eta-\alpha\in M$, i.e, $\eta-\frac{1}{n}\alpha\in \frac{1}{n}M$. Thus, Proposition \ref{reinterpret} applied to  $\beta$ implies $\eta-(\frac{1}{n}\alpha+\beta)=(\eta-\frac{1}{n}\alpha)-\beta\in M$. Therefore, by Proposition \ref{reinterpret} we have $\frac{1}{n}\alpha+\beta\in V_{mn}(M)$.

We now show (3). Let $\alpha,\beta\in V_m(M)$. We first show that $(\alpha+M)\cap (\beta + M)=\emptyset$ for $\alpha\neq \beta$, we proceed by contradiction. Set $W=(\alpha+M)\cap (\beta + M)$, and  note that $W\subseteq (\alpha+M)$ splits since
$$(\alpha+M)\setminus W=
(\alpha+M)\cap \left(\frac{1}{m} M\setminus (\beta+M)\right)
$$
is an $M$-module (see Remark \ref{remDefVm}). Let $T=(\alpha+M)\setminus W$. If $\gamma\in T$ and $z\in W-\alpha\subseteq M$, then $\gamma+z\in T$ but also  $\gamma+z\in \gamma +(W-\alpha)\subseteq W+M=W$, which is not possible.
We conclude $T=\emptyset$, and so $W= (\alpha+M)$. Thus, $\alpha+M\subseteq \beta +M$. By symmetry, $\beta+M\subseteq \alpha +M$, an then $\beta+M= \alpha +M$. It follows that $\alpha-\beta\in M\cap V_m(M)=\{0\}$ by Proposition \ref{minusmbelong} and Remark \ref{mnotvm}, which is a contradiction. Therefore,  the union 
$\bigcup_{\alpha\in V_m(M)} (\alpha + M)\subseteq \frac{1}{m}M$ is disjoint, and so,
$\bigcup_{\alpha\in V_m(M)} \left(\frac{\alpha}{n} + \frac{1}{n}M\right)\subset\frac{1}{mn}M$ is also disjoint.
By applying the same argument, the union
$$
\bigcup_{\alpha\in V_m(M)} \bigcup_{\beta\in V_n(M)} \left(\frac{\alpha}{n} + \beta +M\right) 
\subset\frac{1}{mn}M$$
is disjoint. Hence, $$\frac{1}{n}V_m(M)\oplus V_n(M)= \left\{ \frac{\alpha}{n} + \beta\mid {\alpha\in V_m(M)},{\beta\in V_n(M)} \right\} \subseteq \frac{1}{mn}M$$ is a set of of cardinality 
$|V_m(M)|\cdot |V_n(M)|.$ Finally, the inequality follows from Part (2).
\end{proof}

\begin{proposition}\label{vmfinite}
Let $m\in \mathcal{A}(M)$ and $\alpha \in \frac{1}{m}M$. Write $\alpha = \frac{c_1}{m}\gamma_1+\cdots +\frac{c_u}{m}\gamma_u$ with $c_1,\ldots, c_u\in \NN$.
If $c_1+\ldots+ c_u\gs (m-1)u+1$, then $\alpha\not \in V_m(M)$. 
Furthermore, $|V_m (M)|\leq m^{\rk(M)}$.
\end{proposition}
\begin{proof}
By assumption we have that $\alpha \in \gamma_i + \frac{1}{m}M$ for some $1\ls i\ls u$. 
By way of contradiction suppose that $\alpha\in V_m(M)$.  Since $\alpha-\gamma_i\in \frac{1}{m}M$, it follows by Proposition \ref{minusmbelong}  that  $\gamma_i \in V_m(M)$. However, this contradicts Remark \ref{mnotvm}, and therefore $\alpha\notin V_m(M)$.

For the second claim, recall that 
$\bigcup_{\alpha \in V_m(M)} (\alpha+M)$
is a disjoint union of $M$-modules (see proof of Lemma \ref{vminvmn} (3)), and thus $\bigcup_{\alpha \in V_m(M)} (\alpha+M) \subseteq \frac{1}{m}M$ is pure.
As a consequence, the $\ZZ M$-module $\ZZ\left(\frac{1}{m} M\right)$ contains $\bigoplus_{\alpha \in V_m(M)} \ZZ(\alpha+M)$ as a free direct summand, and thus $|V_m (M)| \leq m^{\rk(M)}$, where the latter is the rank of $\ZZ\left(\frac{1}{m}M\right)$ as a $\ZZ M$-module.
\end{proof}

We now define a new numerical invariant for seminormal monoid. This number plays an important role in our main results. 
This invariant is inspired by the $F$-pure threshold of a ring \cite{TW2004}. This is because the $F$-pure threshold of a standard graded algebra can be described as the supremum among the degrees of a minimal generator of a free summand of $R^{1/p^e}$ \cite{DSNB}. However, the $F$-pure threshold of $R=\kk[M]$ can be different than the  pure threshold of $M$ (see Example \ref{ExVeronese} and Remark \ref{remark mpt VS fpt}).
 In Proposition \ref{mstRat} we prove that for normal monoids this invariant is rational.

\begin{definition}\label{mstDef}
We define the {\it pure threshold} of $M$ as
$$
\mpt(M)=\sup\{|\alpha|\; |\; \alpha\in V(M)\}.
$$
If $V(M)=\emptyset$, i.e., if $M$ is not seminormal, we set $\mpt(M)=-\infty$.
\end{definition}

\begin{remark}\label{sstfinite}
Let $b=\max\{|\gamma_1|,\ldots, |\gamma_u|\}$. By Proposition \ref{vmfinite} we have $|\alpha|< bu$ for every $\alpha\in V(M)$. Therefore, $\mpt(M)<\infty$.
\end{remark}

We now discuss how the pure threshold of a monoid $M$ can be obtained from any increasing sequence in $\cA(M)$.

\begin{proposition}\label{mtinfty}
Let $\{m_t \}_{t\in \ZZ_{\gs 1}}$ be the elements of   $\cA(M)$  ordered increasingly. Then,
$$
\lim\limits_{t\to \infty}\max\{|\alpha| \; |\; \alpha\in V_{m_t}(M)\}=\mpt(M), 
$$
In particular, 
$$
\lim\limits_{t\to \infty}\max\{|\alpha| \; |\; \alpha\in V_{m^t}(M)\}=\mpt(M) 
$$
for any $m\in\cA(M)$.
\end{proposition}
\begin{proof}
If $\mpt(M)=0$, the result follows. We assume  $\mpt(M)>0$.  Let $b=\max\{|\gamma_1|,\ldots, |\gamma_u|\}$ and for any $n\in \mathcal{A}(M)$ set $a_n=\max\{|\alpha| \; |\; \alpha\in V_{n}(M)\}$. Fix $\epsilon>0$ and  $N\in \NN$ such that $\frac{bu}{m_N}<\frac{\epsilon}{2}$. 
  Let $m'\in \cA(M)$ be such that  $b_{m'}>\mpt(M)-\frac{\epsilon}{2}$ and fix $t\gs N$. By Lemma \ref{vminvmn} we  have $b_{m_tm'}>\mpt(M)-\frac{\epsilon}{2}$. 
Consider $\alpha =\frac{c_1}{m_t m'}\gamma_1 +\cdots +\frac{c_u}{m_t m'}\gamma_u\in V_{m_tm'}(M)$ with $c_i\in \ZZ_{\gs 0}$ and $|\alpha|=b_{m_tm'}$. For each $1\ls i\ls u$ let $0\ls r_i<m'$ be such that $c_i\equiv r_i \pmod{m'}.$ By Proposition \ref{minusmbelong} and Lemma \ref{vminvmn} (1) we have 
$$\alpha':=\frac{c_1-r_1}{m_tm'}\gamma_1 +\cdots +\frac{c_u-r_u}{m_tm'}\gamma_u\in  \frac{1}{m_t}M\cap V_{m_tm'}(M)=V_{m_t}(M).$$
Moreover, 
$$\mpt(M)-|\alpha'|=\mpt(M)-|\alpha|+(|\alpha|-|\alpha'|)<\frac{\epsilon}{2}+\frac{bu}{m_t}<\epsilon.$$
Since $\epsilon$ was chosen arbitrarily, the result follows.
\end{proof}

We now compute some examples of pure thresholds. We note that this invariant depends on the grading given by the embedding $M\subseteq \ZZ^q.$

\begin{example}
Let $M$ be generated by $d_1 e_1,\ldots, d_q e_q\in \ZZ^q$, where $d_i\in\ZZ_{>0}$ and $\{e_1,\ldots,e_q\}$ is the canonical basis in  $\ZZ^q$.
Then, $V_m(M)=\{(d_1 \frac{\alpha_1}{m},\ldots,d_q \frac{\alpha_q}{m})\in  \frac{1}{m}\ZZ^q \; | \;0\leq  \alpha_i\leq m-1\}$, and so, $\mpt(M)=d_1+\ldots+d_q$.
\end{example}

\begin{example}\label{ExVeronese}
Let $q\in \ZZ_{>1}$, $t\in\ZZ_{>0}$ and  $M=\{\alpha\in\ZZ^q_{\gs 0} \; | \; |\alpha|\in t\ZZ_{>0}\}$.
Then $\kk[M]$ is the Veronese subring of order $t$ of a polynomial ring $\kk[x_1,\ldots,x_q]$ with the grading $\deg(x_i) = 1$. We have that
$$
V_m(M)=\left\{\left(\frac{\alpha_1}{m},\ldots,\frac{\alpha_q}{m}\right)\in \frac{1}{m}\ZZ^q \; | \;0\leq  \alpha_i\leq m-1 \hbox{ and }|\alpha|\in t\ZZ_{>0}\right\},
$$ 
and therefore $\mpt(M)=q$. We point out that, if $\kk$ has prime characteristic, then $\fpt(\kk[M]) = \frac{q}{t}$ \cite[Example 6.1]{HWY}. 
\end{example}
\begin{remark} \label{remark mpt VS fpt} It follows from Example \ref{ExVeronese} that $\mpt(M)$ may differ from $\fpt(\kk[M])$ even when $M$ is normal. This is not surprising since $\fpt(\kk[M])$ is independent of the presentation of $\kk[M]$ as a quotient of a polynomial ring, while we have already observed that $\mpt(M)$ heavily depends on the degrees of the generators and on the embedding of $M$.
\end{remark}

The following construction allows us to provide bounds for depths of affine semigroup rings (see Section \ref{SectionRings}). In Proposition \ref{indeedprime} we justify the terminology used in the definition.

\begin{definition}\label{PMdef}
 We define the  {\it pure prime} of $M$  by $$ \cP(M)=M\setminus \displaystyle\bigcup_{m\in \ZZ_{> 1}}mV_m(M).$$
\end{definition}

\begin{proposition}\label{indeedprime} Let $M$ be an affine monoid. Then $\cP(M)$ is a   prime ideal of $M$.
\end{proposition}
\begin{proof}

Since $\cP(M)=  \bigcap_{m\in \ZZ_{m> 1}}\left(M\setminus mV_m(M)\right)$, it follows from Corollary \ref{compIdeal} that $\cP(M)$ is an ideal of $M$. Now, let  $a,b \in M\setminus  \cP(M)$ and $m,n\in \cA(M)$ be such that $\alpha:=\frac{a}{m}\in V_m(M)$ and $\beta:=\frac{b}{m}\in V_n(M)$. We claim that $\frac{a+b}{mn}\in V_{mn}(M)$ which implies $a+b\not\in \cP(M)$, finishing the proof. Indeed,  suppose $\frac{a+b}{mn}\not\in V_{mn}(M)$, then from Proposition \ref{minusmbelong}  it follows that 
$$\frac{1}{n}\alpha +\beta =\frac{a}{mn}+\frac{mb}{mn}= \frac{a+b}{mn}+\frac{(m-1)b}{mn} \not\in V_{mn}(M),$$ 
which contradicts Lemma \ref{vminvmn} (2).
\end{proof}

We obtain the following theorem that relates  $\cP(M)$ with the normality of $M$.

\begin{theorem}\label{allVmNormal}
Let $M$ be an affine monoid. Then $M$ is normal if and only $\cP(M)=\emptyset.$
\end{theorem}
\begin{proof}
We begin with the forward direction. Since $M$ is normal, by  Remark \ref{VmIsom} we can assume that $M$  is a  submonoid of $\NN^n$ for some  $n\in \ZZ
_{>0}$ and such that $M=\ZZ M \cap \NN^n$ \cite[Theorem 2.29]{bruns2009polytopes}. Fix $a\in M\subseteq \NN^n$ and chose $m\in \ZZ_{> 1}$ bigger than every entry in  $a$.  By Remark \ref{allmNorm} we have $m\in \cA(M)$ and $m\in \cA(\NN^n)$. By the choice of $m$, it is clear that 
\begin{equation}\label{interZZ0}
\left(\frac{1}{m}\NN^n-
\frac{a}{m}\right)\cap \ZZ^n\subseteq \NN^n.
\end{equation}
Thus, the left hand side expression in  \eqref{interZZ0} is equal to $\left(\frac{1}{m}\NN^n-
\frac{a}{m}\right) \cap \NN^n$. Intersecting this with $\ZZ M$ we obtain,
\begin{align*}
\left(\frac{1}{m}\NN^n-
\frac{a}{m}\right) \cap \NN^n\cap \ZZ M &=\left(\frac{1}{m}\NN^n-
\frac{a}{m}\right)\cap  M=\left(\frac{1}{m}M-
\frac{a}{m}\right)\cap  M.
\end{align*}
Therefore, $\left(\frac{1}{m}M-
\frac{a}{m}\right)\cap  M\subseteq \NN^n\cap \ZZ M=M,$
which shows $a\in mV_m(M)$.

We continue with the backward direction. Let $f\in \ZZ M$ be such that $nf\in M$ for some $n\in \ZZ_{>1}$. By Proposition \ref{redefnormal} it suffices to show $f\in M$. Write $f=a-b$ with $a,b\in M$, then $na\in nb+M$. Thus, $n(a,b)=b+(n-1)(a,b)$, where $(a,b)$ denotes the ideal of $M$ generated by the set $\{a,b\}$. It follows that $(n+r)(a,b)=rb+n(a,b)$ for every $r\in \ZZ_{>0}$. Hence,
\begin{equation*}\label{intClo}
ra+nb\in (n+r)(a,b)=rb+n(a,b)\subseteq rb + M\quad \text{ for every }  r\in \ZZ_{>0}.
\end{equation*}
By assumption there exists  $m\in \cA(M)$ such that $nb\in  mV_m(M)$. Therefore,
$$f=a-b=\left(\frac{ma+nb}{m}\right)-\frac{nb}{m}-b\in \left(b+\frac{1}{m}M\right)-\frac{nb}{m}-b=\frac{1}{m}M-\frac{nb}{m}.$$
On the other hand, $f\in \ZZ M$, then $f\in \left(\frac{1}{m}M-\frac{nb}{m}\right)\cap \ZZ M=M$,
which finishes the proof. 
\end{proof}

\begin{corollary}\label{FaceP}
There exists a face $\cF_M$ of $C(M)$ such that $M\setminus \cP(M) = M\cap \cF_M$. Moreover, the monoid $M\cap \cF_M$ is normal.
\end{corollary}

\begin{proof}
The first part follows from Proposition \ref{indeedprime} and the correspondence between prime ideals of monoids and faces of their cones \cite[Proposition 2.36]{bruns2009polytopes}. By Theorem \ref{allVmNormal} to show that $M\cap \cF_M$ is normal it suffices to show $M\cap \cF_M=\bigcup_{m\in \ZZ_{> 1}}mV_m(M\cap \cF_M).$ Moreover, we may assume $M$ is seminormal. We note that $$M\cap \cF_M= \bigcup_{m\in \ZZ_{>1}}mV_m(M)\subseteq  \bigcup_{m\in \ZZ_{>1}}mV_m(M\cap \cF_M),$$
where the last inclusion follows from  Remark \ref{VmProjFace}. Since we always have the other inclusion $\bigcup_{m\in \ZZ_{>1}}mV_m(M\cap \cF_M)\subseteq M\cap \cF_M$, the proof is complete.
\end{proof}

The previous proposition allows us to define the following invariant of affine monoids, the pure dimension. As we see in Corollary \ref{msdCor}, this new notion measures how far a monoid is from being normal. In  Theorem \ref{ThmDepthMPT} we use this invariant to provide lower bounds for the depth of affine semigoup rings. 

\begin{definition}\label{msd}
The face $\cF_M$   in Proposition \ref{FaceP} is called the {\it pure prime face} of $M$. We define the {\it pure dimension} of $M$   by $\mpdim(M):=\rank (M \cap \cF_M).$ If $\cF_M=\emptyset$,  we set $\mpdim(M)=-\infty$.
\end{definition}

\begin{corollary}\label{msdCor}
Let $M$ be an affine monoid. Then
\begin{enumerate}
\item $\mpdim(M)\ls \rank(M)$.
\item $\mpdim(M)\gs 0$ if and only if $M$ is seminormal.
\item $\mpdim(M)=\rank(M)$ if and only if  $M$ is normal.
\end{enumerate}
\end{corollary}
\begin{proof}
Part (1) follows directly from the definition. For part (2) we note that $\mpdim(M)\gs 0$ if and only if $\displaystyle\bigcup_{m\in \ZZ_{> 1}}mV_m(M)\neq \emptyset$, and by  Remarks \ref{allVmNormal} and \ref{Vmnotempty} this is equivalent to $M $ being seminormal. Part (3)  follows  from Theorem \ref{allVmNormal}.
\end{proof}

We finish this section with the following example.

\begin{example}\label{goodExample}
Let $M=\ZZ_{\gs 0}(2,0)+\ZZ_{\gs 0}(1,1)+\ZZ_{\gs 0}(0,1)$ be the monoid with set of generators  $\{(2,0), (1,1), (0,1)\}$. It is easy to see that $M=\ZZ_{\gs 0}^2 \setminus \{(2a+1,0)\mid a\in \ZZ_{\gs  0}\}$ (see Figure \ref{monoidEx}). We also have that $\cA(M)=\{m\in \ZZ_{> 1}\mid m \text{ is odd}\}$. Then, for every $m\in \cA(M)$ we have $mV_m(M)=\{(a,0)\mid a \text{ is even and } a<m \}$. Therefore, $\mpt(M)=1$, $\cF_M = \RR_{\gs 0}(1,0)$,  and $\mpdim(M)=1$. In particular, by Corollary \ref{msdCor}, $M$ is seminormal but not normal; this is consistent with \cite[Example 1.0.3]{Li}.
\end{example}

\begin{figure}[ht]
\begin{center}
\includegraphics[width=0.3\textwidth]{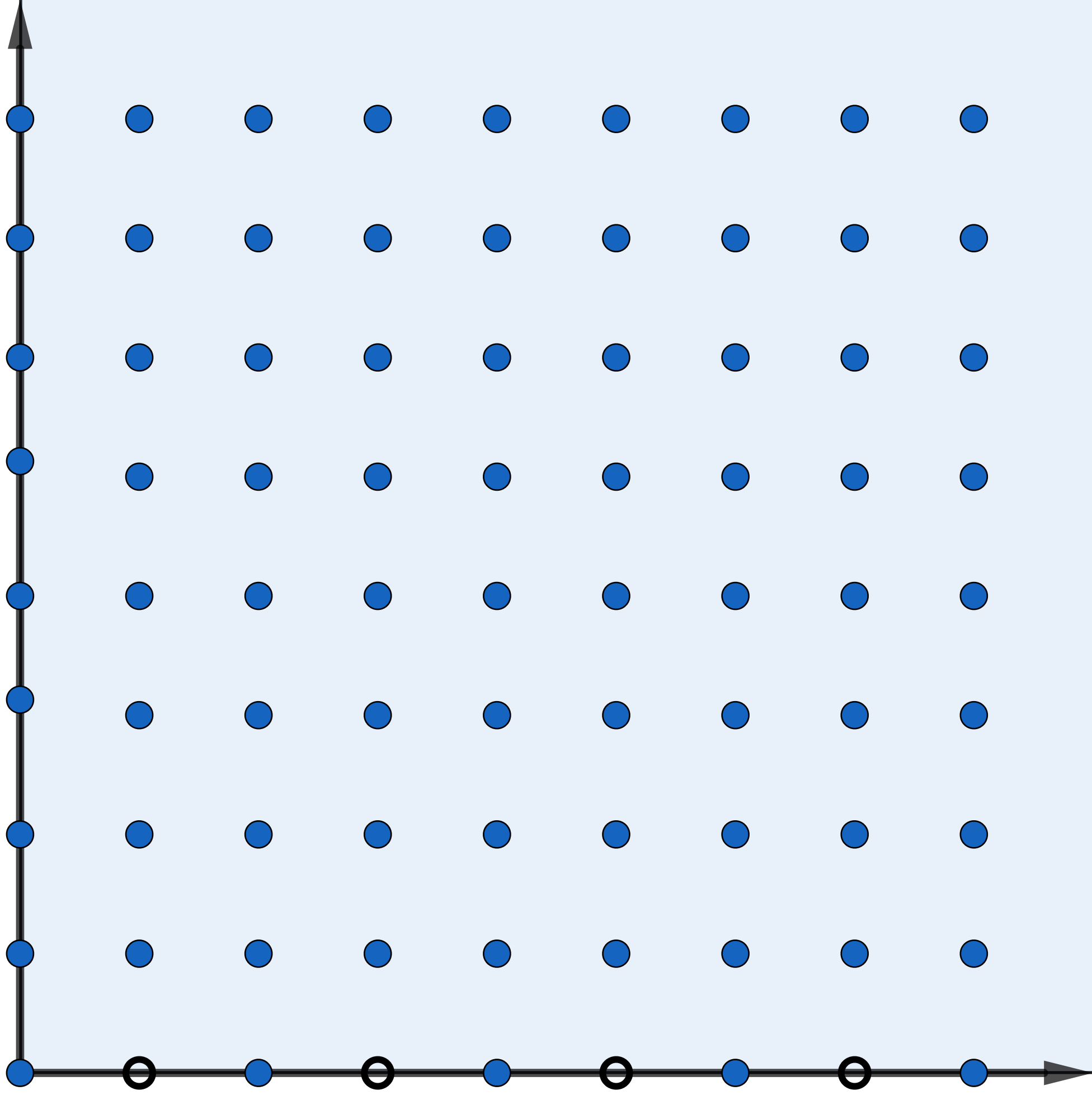}
\end{center}
\caption{ The monoid $M=\ZZ_{\gs 0}(2,0)+\ZZ_{\gs 0}(1,1)+\ZZ_{\gs 0}(0,1)$.}
	\label{monoidEx} 
\end{figure}

\section{Asymptotic growth of number of pure translations}\label{growthSec}

In the  short section, we study the asymptotic behavior of the number of elements in the sets $V_m(M)$. Throughout we adopt the same notation from Section \ref{SplitSection}.

\begin{definition}\label{BMdef}
Let $M$ be a seminormal affine monoid, and let $\cF_M$ be its pure prime face. For every  $m\in \cA(M)$ we define
$$B_m(M)=\bigcup_{\alpha\in V_m(M)}\left((\alpha-\cF_M)\cap \cF_M\right).$$
Moreover, we set
$$B(M)=\bigcup_{m\in \cA(M)}B_m(M).$$
\end{definition}

When $M$ is normal, there is a simple description of $B(M)$ as the region in Notation \ref{regionNorm}. We prove that these regions coincide in Lemma \ref{BMlemma}, which also  includes important properties of  $B(M)$.

\begin{notation}\label{regionNorm}
Let $M$ be a normal affine monoid. Let $\{H_1,\ldots, H_s\}$ be the supporting hyperplanes of $C(M)$ so that $C(M)=H_1^+\cap \cdots \cap H_s^+.$ 
Let $v_i\in \QQ^q$ be rational vectors such that $x\in H_i^+$ if and only if $\langle x,v_i\rangle \gs 0$  for $1\ls i\ls s$. We can further assume that $\langle x,v_i\rangle\in \ZZ$ for every $x\in \ZZ M$ and that $\min\{\langle \alpha,v_i\rangle\mid \alpha\in M\}=1$ \cite[Remark 1.72]{bruns2009polytopes}. We define $\Delta$ by 
$$\Delta = \{x\in \RR^q\mid 0\ls \langle x,v_i\rangle<1 \text{ for } 1\ls i\ls s\}.$$
\end{notation}

\begin{lemma}\label{BMlemma}
Let $M$ be a seminormal affine monoid. Then
\begin{enumerate}
\item $B(M)$ is a bounded set and it has  volume, i.e., its boundary has measure zero in the $\dim(\RR\Fcp)$-dimensional Lebesgue measure on  $\RR \Fcp$.
\item There exists an increasing sequence $\{p_t\}_{t\in \ZZ_{\gs 1}} \subseteq \cA(M)$ such that $B_{p_t}\subseteq B_{p_{t+1}}$ for every $t\in \ZZ_{\gs 1}$ and  $B(M)=\bigcup_{t\in \ZZ_{\gs 1}}B_{p_t}(M).$ 
\item $V_m(M)=\frac{1}{m}M\cap B(M)$ for every $m\in \cA(M)$.
\item If $M$ is normal and $\Delta$ is as in Notation \ref{regionNorm}, then $B(M)=\Delta$.
\end{enumerate}
\end{lemma}
\begin{proof}
We begin with (1). Let $\{g_1, \ldots,g_l\}$ be a minimal set of generators of $M\cap \cF_M$ and consider the region $\Gamma = \cup_{i=1}^l(g_i+\cF_M)$. Let $\alpha\in \frac{1}{m}M\cap \Gamma$, then $\alpha = g_i+\eta$ for some $i$ and $\eta\in \cF_M$.  
Since $M\cap \cF_M$ is normal by Corollary \ref{FaceP}, it follows that $\eta\in   \frac{\ZZ}{m}(M\cap \cF_M)\cap \cF_M\subseteq \frac{\ZZ}{m}M$.  If $\alpha\in V_m(M)$, then Proposition \ref{minusmbelong} implies  $g_i\in V_m(M)$ which contradicts Remark \ref{mnotvm}. We conclude $V_m(M)$, and then $B_m(M)$, is contained in $\cF_M\setminus \Gamma$ which is bounded. 

Now, let $\partial$ and $^\circ$ denote boundary and interior on $\RR\Fcp$, respectively. Let $\mu$ denote the $\dim(\RR\Fcp)$-dimensional Lebesgue measure on $\RR \Fcp$. We note that for any  $x\in \partial B(M)\setminus \partial \Fcp$ we have $x+\cF_M^\circ \subseteq \Fcp\setminus \overline{B(M)}$; indeed, if $x+y\in \overline{B(M)}$ for some $y\in \Fcp^\circ$, then $x+y'\in B(M)$ for some $y'\in \Fcp^\circ$, which would imply $x\in B(M)^\circ.$ 
Therefore, for any $r>0$  and any $x\in \partial B(M)\setminus \partial \Fcp$ we have $\partial B(M)\cap B(r,x)\cap (x+\Fcp^\circ) =\emptyset$, where  $B(r,x)$ denotes the ball in $\RR\Fcp$ with radius $r$ and center $x$. Therefore, there exists  a real $c<1 $ such that for any such $r$ and $x$ we have $\frac{\mu(\partial B(M)\cap B(r,x))}{\mu(B(r,x))}<c.$ By Lebesgue's density theorem \cite[Corollary 2.14]{mattila}, we conclude $\mu(\partial B(M))=0$.

We continue with (2). Let $\{m_t \}_{t\in \ZZ_{\gs 1}}$ be the elements of   $\cA(M)$  ordered increasingly. For each $t\in \ZZ_{\gs 1}$ set $p_t=m_1\cdots m_t$ and notice $p_t\in \cA(M)$ by Lemma \ref{mnboth}. The conclusion now follows from Lemma \ref{mnotvm} (1).

Now we prove (3). Let $m_t\in \cA(M)$ and $\alpha \in\frac{1}{m_t}M\cap B(M)$, it suffices to show $\alpha\in V_{m_t}(M)$. By (2), we have $\alpha\in B_{p_i}(M)$ for some $i$. We may assume $i\gs t$ and then $m_t$ divides $p_i$. Therefore, there exists $\eta\in \cF_M$ and $\beta\in V_{p_i}(M)$ such that $\alpha+\eta = \beta$. Since $M\cap \cF_M$ is normal by Corollary \ref{FaceP}, it follows that $\eta\in \frac{\ZZ}{p_i}(M\cap \cF_M)\cap \cF_M\subseteq \frac{\ZZ}{p_i}M$. Thus, $\alpha\in \frac{1}{m_t}M\cap V_{p_i}(M)=V_{m_t}(M)$ by Proposition \ref{minusmbelong} and Lemma \ref{vminvmn}, which finishes the proof.

We finish with (4). If $M$ is normal we have $V_m(M)= \frac{\ZZ}{m}M\cap \Delta$  \cite[Lemma 3.11]{von2011f}. Thus, the equality $B(M)=\Delta$ follows as the set $\cup_{m\in \ZZ_{\gs 1}} \frac{\ZZ}{m}M$ is dense in $\RR^q$.
\end{proof}

From Lemma \ref{BMlemma} (4) we obtain that the pure threshold of normal monoids is rational.

\begin{proposition}\label{mstRat}
If $M$ is normal, then $\mpt(M)\in \QQ_{>0}$.
\end{proposition}
\begin{proof}
The statement follows readily from Lemma \ref{BMlemma} (4) and the equality $\mpt(M)=\sup\{|\alpha|\mid \alpha\in  B(M)\}$.
\end{proof}

We now turn our focus to asymptotic growth of the number of elements in the sets $V_m(M)$. We define the following limit, which we prove exists in Theorem \ref{ratioExists}

\begin{definition}\label{msrDef}
Let $M$ be a seminormal affine monoid. Set $s=\mpdim(M)$ and let $\{m_t \}_{t\in \ZZ_{\gs 1}}$ be the elements of   $\cA(M)$  ordered increasingly. We define the {\it   pure ratio} of  $M$ as
$$
\mpr(M)=\lim\limits_{t\to\infty} \frac{|V_{m_t}(M)|}{m_t^s}.
$$
We define the {\it  pure signature} of  $M$ as
$$
\mps(M)=\lim\limits_{t\to\infty} \frac{|V_{m_t}(M)|}{m_t^{\rk(M)}}.
$$
\end{definition}

In the following theorem we show that $\mpr(M)$ exists as a limit, and that it equals the relative volume of $B(M)$.
 Here,   by {\it relative volume} with respect to a  lattice $L\subseteq H$ of rank $r$ in an $r$-dimensional hyperplane $H\subseteq \RR^q$, denoted by $\vol_L$, we mean the $r$-dimensional volume in $H$ normalized such that any fundamental domain of $L$ has volume one.

\begin{theorem}\label{ratioExists}
Let $M$ be a seminormal affine monoid. We have that
$$
\mpr(M)=\vol_{\ZZ(M\cap \Fcp)}(B(M))>0.
$$
In particular,
  $M$ is normal if and only if $\mps(M)>0$. Furthermore, in this case
   $\mps(M)\in \QQ_{>0}$.
\end{theorem}
\begin{proof}
By Lemma \ref{BMlemma} (1), the  characteristic function  $\chi_{B(M)}$ is   Riemann integrable. Now, by Lemma \ref{BMlemma} (3) and Corollary \ref{FaceP} we have $V_{m_t}(M)=\frac{\ZZ}{m_t}(M\cap \Fcp)\cap B(M)$. Thus, $\frac{|V_{m_t}(M)|}{m_t^s}$ is a Riemann sum for  $\chi_{B(M)}$ with normalized volumes of the cells and mesh the diameter of a fundamental domain for $\frac{\ZZ(M\cap \Fcp)}{m_t}$.  Therefore, by taking the limit $t\to \infty$ we obtain that the limit exists and is equal to  $\vol_{\ZZ(M\cap \Fcp)}(B(M))$.  We note that $\vol_{\ZZ(M\cap \Fcp)}(B(M))$ is positive since $V_m(M)$ has interior points of $\Fcp$ (see Corollary \ref{FaceP}). The last statements follow from Corollary \ref{msdCor} and  Lemma \ref{BMlemma} (4).
\end{proof}

Theorem \ref{ratioExists} is related to previous computations done for the $F$-signature of normal semigroup rings \cite{ToricSingh,von2011f}.

\begin{example}
Let $M$ be as in Example \ref{goodExample}. We observe that $|V_m(M)|=\lceil\frac{m}{2}\rceil$ for every $m\in \cA(M)$. Then, $\mpr(M)=\frac{1}{2}$. We also have $B(M)=[0,1)$, therefore $\vol_{\ZZ(M\cap \Fcp)}(B(M))=\frac{1}{2}$, which is consistent with Theorem \ref{ratioExists}.
\end{example}

We end this section with a question motivated by Proposition \ref{mstRat}. This question is open, to the best of our knowledge, for seminormal monoids that are not normal.

\begin{question}\label{Question}
Let $M$ be a seminormal affine monoid. Is $\mpr(M)$ a rational number?
\end{question}

\section{Applications to affine semigroup rings}\label{SectionRings}

Throughout this section we adopt the following notation.
\begin{notation}\label{NotationSM}
Given an affine monoid as in Notation \ref{notation}, we let $R=\kk[M]=K[\bx^\balpha \mid \balpha\in M]\subseteq \kk[\bx]:=\kk[x_1\ldots, x_q]$ be  the
  {\it affine semigroup ring} of $M$.
 Given $m\in \ZZ_{>0}$, we set
$R^{1/m}=\kk\left[\frac{1}{m}M\right]$ the $\kk$-algebra $\kk\left[x^\balpha | \balpha\in \frac{1}{m}M\right]$. 
Given $m\in\ZZ_{>0}$ and $\alpha\in \frac{1}{m} M$, we consider   
$\phi_{\alpha}^m :R^{1/m}\to R$ the $\kk$-linear map given by 
$\phi_{\alpha}^m(\bx^{\beta})=\bx^{\beta -\alpha}$ if $\beta-\alpha\in M$ and zero otherwise. For an ideal $I\subseteq M$, we denote by $\bx^I$ the corresponding  $M$-homogeneous $R$-ideal, 
$$\bx^I=(\bx^\alpha\mid \alpha\in I).$$
For an  $M$-homogeneous element $f=\bx^\alpha\in R$, we denote by $\log(f)=\alpha\in M$ the corresponding element in $M$.

\end{notation}

\begin{remark}
We note that $R\cong R^{1/m}$ via the $\kk$-algebra map given by $\bx^\alpha\mapsto \bx^{\alpha/m}$.
\end{remark}

\begin{proposition}\label{Vmsplit}
Let $M$ be an affine monoid, and let $\alpha \in \frac{1}{m} M$. Then, $\phi_{\alpha}^m$ is a map of $R$-modules if and only if $\alpha\in V_m(M)$.
\end{proposition}
\begin{proof}
We note that $\phi_{\alpha}^m$ is a map of $R$-modules if and only if for every $\gamma \in M$ and $\beta\in \frac{1}{m}M$ we have 
$\phi_{\alpha}^m(\bx^\gamma \bx^\beta)=\bx^\gamma \phi_{\alpha}^m( \bx^\beta)$. 
By the definition of $\phi_\alpha^m$ these are equivalent  to $\phi_{\alpha}^m(\bx^\gamma \bx^\beta)\neq 0$  implies $\bx^\gamma \phi_{\alpha}^m( \bx^\beta)\neq 0$, or equivalently to, 
$$
\gamma +\beta -\alpha \in M \text{ implies }\beta -\alpha \in M.
$$
The conclusion now follows from Proposition \ref{reinterpret}.
\end{proof}

In the next result, we use the semigroup splitting threshold to provide a bound for the Castelnuovo-Mumford regularity of  affine semigoup rings. We refer the reader to Section \ref{Background} for information about $a$-invariants and regularity.

\begin{theorem}\label{ThmRegMPT}
Let $M$ and $R$ be as in Notation \ref{NotationSM}. Then, $a_i(R)\leq -\mpt(M)$. As a consequence, 
\[
\reg(R)\leq \dim(R) - \mpt(M)= \rank(M) - \mpt(M).
\]
Moreover, if we present $R$ as $S/I$, where $S=\kk[x_1,\ldots,x_u]$ and each $x_i$ has degree $d_i:=\deg(x_i)=|\gamma_i|$ the degree of $\gamma_i$ for $i=1,\ldots,u$,   and $I \subseteq S$ is a homogeneous ideal, then  
\[
\beta(I) \leq  \dim(R)+ \sum_{i=1}^u(d_i-1)- \mpt(M)-1 = \rank(M) + \sum_{i=1}^u(d_i-1)- \mpt(M)-1.
\] 
\end{theorem}
\begin{proof}
We can assume that $M$ is seminormal. Let $m\in \cA(M)$ be such that $m>1$, which exists by Proposition \ref{redefnormal} (1). Fix $t\in \NN$ and $\alpha\in V_{m^t}(M)$. From Proposition \ref{Vmsplit} it follows that $\phi^{m^t}_\alpha$ gives a splitting of the homogeneous injective map $R(-|\alpha|)\hookrightarrow R^{1/m^t}$ defined as multiplication by $\bx^\alpha$. Thus, for each $i$, the induced map $$H_{\m}^i(R)(-|\alpha|)=H_{\m}^i(R(-|\alpha|))\to H_{\m}^i(R^{1/m^t})=H_{\m}^i(R)^{1/m^t}$$  
also splits.  By comparing the highest degrees of these modules we obtain 
$$a_i(R)+|\alpha|\ls \frac{a_i(R)}{m^t}.$$
By taking the maximum value of $|\alpha|$ over all $\alpha \in V_{m^t}(M)$ and letting $t\to \infty$, by Proposition \ref{mtinfty}  we obtain that $a_i(R)\ls -\mpt(M)$ as desired. The  inequality for regularity follows by definition, and the last equality by the relation between the rank of  semigroups and dimension of  semigroup rings (see e.g. \cite[p.257]{BrHe}). 

Finally, the inequalities involving $\beta(I)$ follow at once from the fact that $\beta(I) \leq \Reg(I) \leq \Reg(R)-1 = \reg(R)+\sum_{i=1}^u (d_i-1) - 1$ by Lemma \ref{Lemma regularity} and the previous inequalities.
\end{proof}

We now compute the pure threshold for a normal monoid that is Gorenstein. This follows previous work done for the $F$-pure threshold \cite[Theorem B]{DSNB}, which was motivated by a conjecture posted by Hirose, Watanabe and Yoshida \cite{HWY}.
\begin{theorem}
Assume that $M$ is normal of rank $d$. If $R$ is Gorenstein, then $\mpt(M) = -a_d(R)$.
\end{theorem}  
\begin{proof} 
Since $M$ is normal, the Gorenstein property of $R = \kk[M]$ is independent of the field $\kk$ \cite[Remark 6.34]{bruns2009polytopes}. The pure threshold  $\mpt(M)$ is also independent of $\kk$. If $\kk$ has characteristic zero, then $a_d(\kk[M]) = a_d(\QQ[M]) = a_d(\FF_p[M])$ for all $p \gg 0$ (see for instance by \cite[Lemma 4.3]{DDSM} adapted to the positively graded case). If $\LL$ is any field extension of $\kk$, and $\m$ is the homogeneous maximal ideal of $R$, then we have graded isomorphisms $H^d_\m(R) \otimes_\kk \LL \cong H^d_\m(R\otimes_\kk \LL) = H^d_\m(\LL[M])$. Thus, we may assume that $\kk$ is a perfect field of characteristic $p>0$. We can write $R = S/I$, where $S=\kk[T_1,\ldots,T_u]$, each $T_i$ maps to a generator $\bx^{\gamma_i}$ of $R$ and $\deg(T_i) = |\gamma_i| = d_i>0$. Since $R$ is Gorenstein, we have that $\Hom_R(R^{1/p^e},R) \cong (I^{[p^e]}:_S I)/I^{[p^e]} = (f_e + I^{[p^e]})/I^{[p^e]}$ \cite{FedderFputityFsing}. If $\mathbb{F}_\bullet: 0 \to F_c \to \ldots \to F_0=S \to R \to 0$ is a minimal free resolution of $R$ over $S$, then $c=\Ht(I)$ and $F_c = S(-D-a_d(R))$, where $D=\sum_{i=1}^u d_i$. The minimal free resolution of $\mathbb{F}_\bullet^{e}: 0 \to F_c^e \to \ldots \to F_0^e=S \to S/I^{[p^e]} \to 0$ of $S/I^{[p^e]}$ is such that $F_c^e=S(p^e(-D-a_d(R)))$. The comparison map $\mathbb{F}_\bullet^e \to \mathbb{F}_\bullet$ induced by the natural surjection $S/I^{[p^e]} \to R$ in homological degree $c$ is $S(p^e(-D-a_d(R))) \to S(-D-a_d(R))$. Furthermore, it is  given, up to an invertible element, by multiplication by $f_e$  \cite[Lemma 1]{VraciuGorTightClosure}. Since such a map is homogeneous of degree zero, we conclude that $\deg(f_e) = (p^e-1)(D+a_d(R))$. Let $\n=(T_1,\ldots,T_u)$. As $f_e \notin \n^{[p^e]}$ by Fedder's criterion \cite{FedderFputityFsing}, there is a monomial $T_1^{n_1} \cdots T_u^{n_u}$ in its support with $0 \leq n_i \leq p^e-1$ for all $i$. This implies that the map $S/I \to (S/I)^{1/p^e}$ sending $1 \mapsto (T_1^{p^e-1-n_1} \cdots T_u^{p^e-1-n_u})^{1/p^e}$ splits. Via the isomorphism $S/I \cong R$, this means that the map $R \to R^{1/p^e} = \kk[M^{1/p^e}]$ sending $1 \mapsto \left(\bx^{\gamma_1(p^e-1-n_1)} \cdots \bx^{\gamma_u(p^e-1-n_u)}\right)^{1/p^e}=\bx^{\beta(e)}$ splits, and so, $\beta(e) \in V_{p^e}(M)$. Note that 
\[
|\beta(e)| = \frac{\sum_{i=1}^u |\gamma_i|(p^e-1-n_i)}{p^e} = \frac{(p^e-1)D - \sum_{i=1}^u n_id_i}{p^e} = \frac{(p^e-1)D - \deg(f_e)}{p^e} = -a_d(R) \frac{p^e-1}{p^e}.
\]
We have  that $\mpt(M) \geq \lim\limits_{e \to \infty} |\beta(e)| = -a_d(R)$ by Proposition \ref{mtinfty}. As the other inequality always holds by Theorem \ref{ThmRegMPT}, we have equality.
\end{proof}

From the previous result, one may wonder if the converse is true. In particular, as $\fpt(\kk[M])=a_d(\kk[M])$ implies that $\kk[M]$ is Gorenstein if $\kk[M]$ has a structure of standard graded $\kk$-algebra \cite{STV}. This motivates the following question.

\begin{question}
Assume that $M$ is normal of rank $d$. If $\mpt(M) = -a_d(R)$, is  $R$ is Gorenstein?
\end{question} 

We now provide a bound for the depth of $R$, which recovers Hochster's result that normal semigroup rings are Cohen-Macaulay \cite[Theorem 1]{HochsterToric}.

\begin{theorem}\label{ThmDepthMPT}
Let $M$ and $R$ be as in Notation \ref{NotationSM}. Then, 
$\mpdim(M)\leq \Depth(R).$
\end{theorem}
\begin{proof}
We can assume that $M$ is seminormal. Let $S=\kk[y_1,\ldots,y_u]$ endowed with the $M$-grading given by $\deg(y_i)=\gamma_i$.
We set a surjection of $\kk$-algebras $\rho:S\to R$ by $y_i\mapsto \bx^{\gamma_i}$.
Let $\m=(\bx^{\gamma_1},\ldots,\bx^{\gamma_u})\subseteq R$ and $\eta=(y_1,\ldots,y_u)\subseteq S$. We note that $\rho(\eta)=\m.$ Set $J=\Ker(\rho)$, so that $R=S/J$.

Set $t=\Depth(R) =\min\{i\; |\; \Ext^{u-i}_S (R,S)\neq 0\}$. We first show that $\Ann_R \left(\Ext^{u-t}_S(R,S)\right)\subseteq \bx^{\cP(M)}$; we proceed  by contradiction.
Suppose that there exists an $M$-homogeneous element $f\in R$ such that  $f\in \Ann_R \left(\Ext_S^{u-t}(R,S)\right)\setminus  \bx^{\cP(M)}$.
Let $m\in \NN$ be such that $\log(f)\in mV_m(M)$.
Since the multiplication map $\Ext_S^{u-t}(R,S)\FDer{f}\Ext_S^{u-t}(R,S)$ is the zero map, we have that 
$H^t_\m(R)\FDer{f} H^t_\m(R)$ is the zero map by Matlis duality \cite{Matlis}. 
Thus,
$H^t_\m(R^{1/m})\FDer{f^{1/m}} H^t_\m(R^{1/m})$ is the zero map as well. Since the  composition of $R\FDer{\iota} R^{1/m}\FDer{f^{1/m}} R^{1/m}\FDer{\phi_{\log(f)}^m} R$ is the identity, we have the same for the composition
$$
H^{t}_\m(R)\FDer{\iota} H^t_\m(R^{1/m})\FDer{f^{1/m}} H^t_\m(R^{1/m})\FDer{\phi_{\log(f)}^m}H^t_\m(R).
$$
Since the middle map is zero, we have that $H^{t}_\m(R)=0$, which is not possible because $t=\Depth(R).$

Since $\Ann_R \left(\Ext^{u-t}_S(R,S)\right)\subseteq  \bx^{\cP(M)}$, we have that 
\begin{equation}\label{EqDepthSdim1}
\dim \left(\Ext^{u-t}_S(R,S)\right)\geq  \dim \left(R/\bx^{\cP(M)}\right)=\mpdim(M).
\end{equation}
Since $S$ is a Gorenstein ring, its injective resolutions as $S$-module is given by
$$
0\to S \to E^0\to E^2\to \ldots E^u,
$$
where $E^j=\bigoplus_{\Ht(\fp)=j} E(S/\fp)$ is the direct sum of the injective hulls of all the prime ideals in $S$ of height $j$ \cite{Bass}. Therefore, 
\begin{equation}\label{EqDepthSdim2}
\dim \left(\Ext^{u-t}_S(R,S)\right)\leq t.
\end{equation}
Combining Inequalities \eqref{EqDepthSdim1} and \eqref{EqDepthSdim2} we obtain the desired result.
\end{proof}

We now relate the pure ratio of a monoid $M$  to the splitting ratio of $R$ \cite{AE} and $F$-signature \cite{SmithVdB,HunekeLeuschke,Tucker}.
\begin{proposition}\label{Prop MPR SR}
Let $M$ and $R$ be as in Notation \ref{NotationSM}. 
If $\Char(\kk)\in\cA(M)$, then $\mpr(M)$ is the $F$-splitting ratio of $R$. As a consequence, 
if $\Char(\kk)$ is a prime number and $M$ is normal, then $\mps(M)$ equal to the $F$-signature of $R$.
\end{proposition}
\begin{proof}
Let $p=\Char(\kk)$ and $\m$ be the maximal homogeneous ideal in $R$. 
Let 
$$
I_e= \{ f\in R \; |\; \phi(f^{1/p^e})\in\m \; \; \forall\phi\in\Hom(R^{1/p^e},R)\}.
$$
We note that $\dim_\kk(R/I_e)=|V_{p^e}(M)|$, and that  $\bx^{\cP(M)}$ is the splitting prime of $R$ \cite{AE}.
It follows that $\mpdim(M)=\sdim(R)$, and  the result follows for the ratios.

We now discuss the claim about $F$-signature. We have that  $\Char(\kk)\in\cA(M)$ 
 for normal monoids by Proposition \ref{redefnormal}.
The result follows because the
 $F$-signature coincides with the $F$-splitting ratio for strongly $F$-regular rings, and $R$ is strongly F-regular if and only if $M$ is normal.
\end{proof}

We end this section with a monoid version of Kunz's characterization of regularity \cite{Kunz}.
\begin{theorem}\label{ThmKunz}
Let $M$ be an affine monoid. Then, $|V_m(M)|=m^{\rk(M)}$ for some $m\in\ZZ_{>0}$ if and only if $M\cong \ZZ^t_{>0}$ for some $t$.
\end{theorem}
\begin{proof}
Since $|V_m(M)|=m^{\rk(M)}$, we have that $|V_{m^t} (M)|=m^{t\rk(M)}$ 
by Lemma \ref{vminvmn} (3).
Then, 
$$
\mps(R)=\lim\limits_{t\to \infty}\frac{|V_m(M)|}{m^{\rk(M)}}=1.
$$
Hence,
$\mpdim(M)=\rk(M)$, and so, $M$ is a normal monoid by Corollary \ref{msdCor}.
We have that  $\mps(R)=1$.
Then,  $\FF_p[M]$ is a regular graded $\FF_p$-algebra, because $s(R)=1$ by Proposition \ref{Prop MPR SR} and the characterization of regular rings via $F$-signature \cite[Corollary 16]{HunekeLeuschke}.
Moreover, $M$ has a set of $\rk(M)$ minimal generators. 
Hence, $M\cong \ZZ^t_{>0}$.
\end{proof}

\section*{Acknowledgments}
We thank the reviewer for the careful reading of our paper and for the suggested improvements. The first author was partially supported by the PRIN 2020 project 2020355B8Y ``Squarefree Gröbner degenerations, special varieties and related topics". The second author was  supported by NSF Grant DMS \#2001645/2303605. The third author was  supported by CONACyT Grant \#284598.

\bibliographystyle{alpha}
\bibliography{References}

\end{document}